\documentclass[11pt]{amsart}

\usepackage{amsfonts, amssymb, amsmath, graphicx}
\usepackage[all, knot]{xy}

\makeatletter
\@namedef{subjclassname@2020}{%
  \textup{2020} Mathematics Subject Classification}
\makeatother

\newcommand{\N}{\mathbb{N}}
\newcommand{\Z}{\mathbb{Z}}
\newcommand{\R}{\mathbb{R}}

\newcommand{\X}{\mathcal{X}}

\newcommand{\Q}{\mathcal{Q}}
\newcommand{\T}{\mathcal{T}}
\newcommand{\kei}{\triangleright}

\newtheorem{theorem}{Theorem}
\theoremstyle{definition}

\newtheorem{corollary}[theorem]{Corollary}
\newtheorem{proposition}[theorem]{Proposition}

\newtheorem{example}[theorem]{Example}
\newtheorem{remark}[theorem]{Remark}

\def\co{\colon\thinspace}

\begin{document}
\title{Quandle coloring quivers of $(p, 2)$-torus links}

\author{Jagdeep Basi}
\address{Department of Mathematics, California State University, Fresno \linebreak 5245 North Backer Avenue, M/S PB108, CA 93740, USA}
\email{jbasi@protonmail.com}

\author{Carmen Caprau}
\address{Department of Mathematics, California State University, Fresno\linebreak 5245 North Backer Avenue, M/S PB108, CA 93740, USA}
\email{ccaprau@csufresno.edu}

\thanks{CC was partially supported by Simons Foundation grant $\#355640$}

\begin{abstract}
A quandle coloring quiver is a quiver structure, introduced by Karina Cho and Sam Nelson, which is defined on the set of quandle colorings of an oriented knot or link by a finite quandle. We study quandle coloring quivers of $(p, 2)$-torus knots and links with respect to dihedral quandles.
\end{abstract}

\maketitle

\section{Introduction} \label{sec:intro}
A quandle~\cite{DJoyce, Matveev} is a non-associate algebraic structure whose definition is motivated by knot theory, and it provides an axiomatization of the Reidemeister moves for link diagrams. The fundamental quandle of a link is a quandle that can be computed from any diagram of the link.
Joyce~\cite{DJoyce} and Matveev~\cite{Matveev} proved that the fundamental quandle is a stronger invariant of knots and links than the knot group. Quandles have been used to define other useful invariants of oriented links (see~\cite{CEGS, CJKLS, ChoNelson, ElhamdadiNelson, GN} and others). For example, the quandle counting invariant~\cite{ElhamdadiNelson} is the size of the set of all homomorphisms from the fundamental quandle of an oriented link to a finite quandle. More recently, Cho and Nelson~\cite{ChoNelson} introduced the notion of a quandle coloring quiver of an oriented link $K$ with respect to a finite quandle $\X$, which is a directed graph that allows multiple edges and loops, and enhances the quandle counting invariant of $K$ by $\X$.

In this paper, we study quandle coloring quivers of $(p,2)$-torus links by dihedral quandles. The paper is organized as follows. In Sec.~\ref{sec:prelims} we review some basic concepts about torus links and quandles. In Sec.~\ref{ssec:2.2} we also prove a couple of statements about the coloring space of $(p, 2)$-torus links by dihedral quandles, as they will be needed later in the paper. Sec.~\ref{sec:quivers} is the heart of the paper; here we review the definition of quandle coloring quivers of oriented links with respect to finite quandles and study the quivers of $(p, 2)$-torus links with respect to dihedral quandles.

\section{Preliminaries} \label{sec:prelims}

In order to have a self-contained paper, we begin by reviewing relevant definitions and concepts.

\subsection{Basic definitions} \label{sses:2.1}
A knot or link $K$ is called a \textit{$(p, q)$-torus knot} or \textit{link} if it lies, without any points of intersection, on the surface of a trivial torus in $\R^3$. The integers $p$ and $q$ represent the number of times $K$ wraps around the meridian and, respectively, the longitude of the torus. We denote a $(p, q)$-torus knot or link by $\mathcal{T}(p,q)$. It is known that $\mathcal{T}(p,q)$ is a knot if and only if $\gcd(p, q) = 1$ and $p$ and $q$ are not equal to 1 simultaneously, and that it is a link with $d$ components if and only if $\gcd(p, q) = d$. Moreover, $\mathcal{T}(p,q)$ is the trivial knot if and only if either $p$ or $q$ is equal to $1$ or $-1$. We use the term link to refer generically to both knots and links. It is also known that the $(p, q)$-torus link is isotopic to the $(q, p)$-torus link. The $(-p, -q)$-torus link is the $(p, q)$-torus link with reverse orientation, and torus links are invertible. For simplicity, we will assume that $p, q >0$. The mirror image of $\mathcal{T}(p,q)$ is $\mathcal{T}(p, - q)$ and a non-trivial torus-knot is chiral.

It is known that the crossing number of a torus link with $p, q >0$ is given by $c = \text{min}((p-1)q, (q-1)p)$. The link $\mathcal{T}(p,q)$ can also be represented as the closure of the braid $(\sigma_1 \sigma_2 \dots \sigma_{p-1})^q$. Since $\mathcal{T}(p,q)$ is isotopic to $\mathcal{T}(q, p)$, this torus link has also a diagram which is the closure of the braid $(\sigma_1 \sigma_2 \dots \sigma_{q-1})^p$.

The simplest nontrivial torus knot is the trefoil knot, $\mathcal{T}(3,2)$, depicted in Fig.~\ref{Trefoil}. For more details on torus knots and links, we refer the reader to the books~{\cite{Kawauchi, Lickorish, Livingston, Murasugi}. 
\begin{figure}[ht] 
\includegraphics[height=1in]{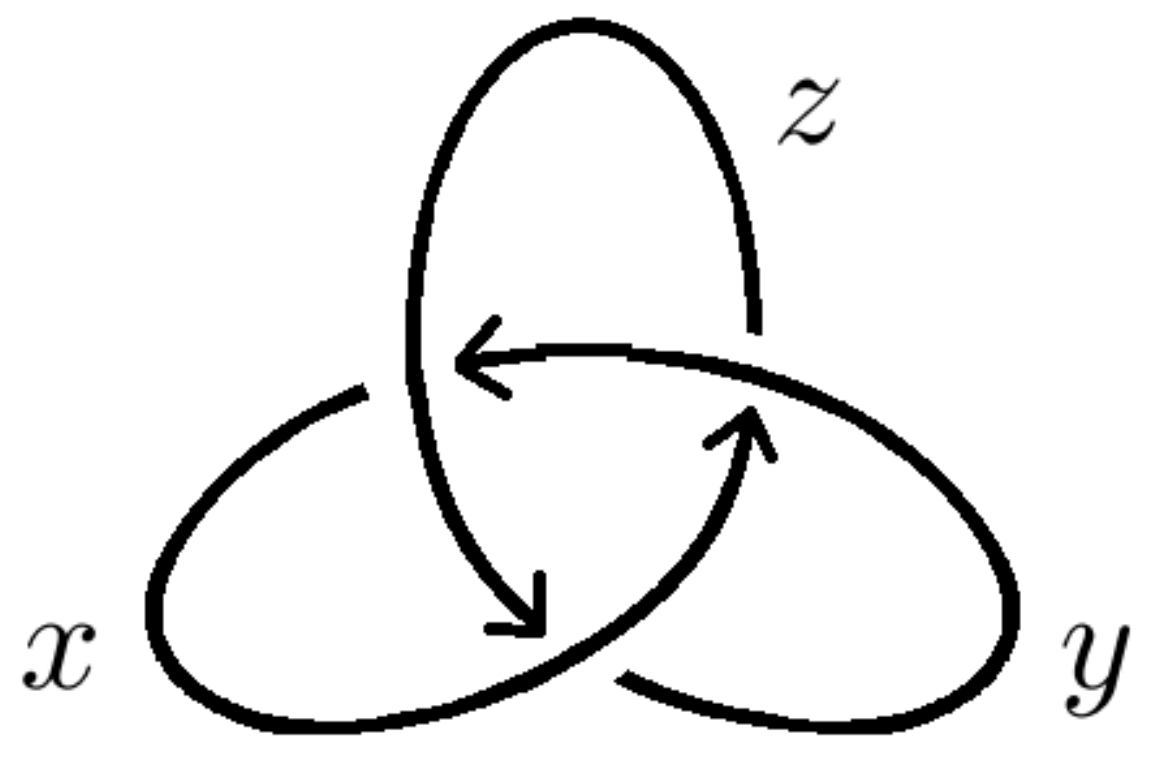}
\caption{Oriented torus knot $\T(3,2)$}
\label{Trefoil}
\end{figure}

The term quandle is attributed to Joyce in his 1982 work~\cite{DJoyce} as an algebraic structure whose axioms are motivated by the Reidemeister moves.

A \textit{quandle}~\cite{DJoyce, Kamada, Matveev} is a non-empty set $\X$ together with a binary operation $\kei: \X \times \X \to \X$ satisfying the following axioms:
\begin{enumerate}
\item For all $x \in \X$, $x \kei x = x$. 
\item For all $y \in \X$, the map $\beta_y: \X \to \X$ defined by $\beta_y(x) = x \kei y$ is invertible. 
\item For all $x,y,z \in \X$, $(x \kei y) \kei z = (x \kei z ) \kei (y \kei z)$.
\end{enumerate}
In this case, we use the notation $(\X,\kei)$ or $(\X,\kei_{\X})$, if we need to be specific about the quandle operation.


A \emph{quandle homomorphism}~\cite{Inoue} between two quandles $(\X,\kei_{\X})$ and $(\mathcal{Y},\kei_{\mathcal{Y}})$ is a map $f \co \X \to \mathcal{Y}$ such that $f(a\kei_{\X} b) = f(a) \kei_{\mathcal{Y}} f(b)$, for any $a,b \in \X$.

It is not hard to see that the composition of quandle homomorphisms is again a quandle homomorphism. We note that the second axiom of a quandle implies that for all $y \in \X$, the map $\beta_y$ is an automorphism of $\X$. Moreover, by the first axiom of a qundle, the map $\beta_y$ fixes $y$.
\begin{example} \label{dihedral quandle2}
Let $n \in \Z, n \ge2, \X = \Z_n = \{0, 1, \dots, n-1\}$ and define $x \kei y = (2y - x) \mod  n$, for all $x,y \in \Z_n$. That is, $x \kei y$ is the remainder of $2y-x$ upon division by $n$. Then $\Z_n$ together with this operation $\kei$ is a quandle, called the \textit{dihedral quandle} (see~\cite{ElhamdadiNelson, MTakasaki}).
\end{example}

\begin{example}\label{fundamental quandle}
To every oriented link diagram of $K$, we can naturally associate its \textit{fundamental quandle} (also referred to as the \textit{knot quandle}), denoted by $Q(K)$. Given a diagram $D$ of $K$, its crossings divide $D$ into arcs and we use a set of labels to label all of these arcs. The fundamental quandle $Q(D)$ is the quandle generated by the set of arc-labels together with the crossing relations given by the operation $\kei$ explained in Fig.~\ref{QuandleCrossing}.
\begin{figure}[ht]
\includegraphics[height = 1.2in]{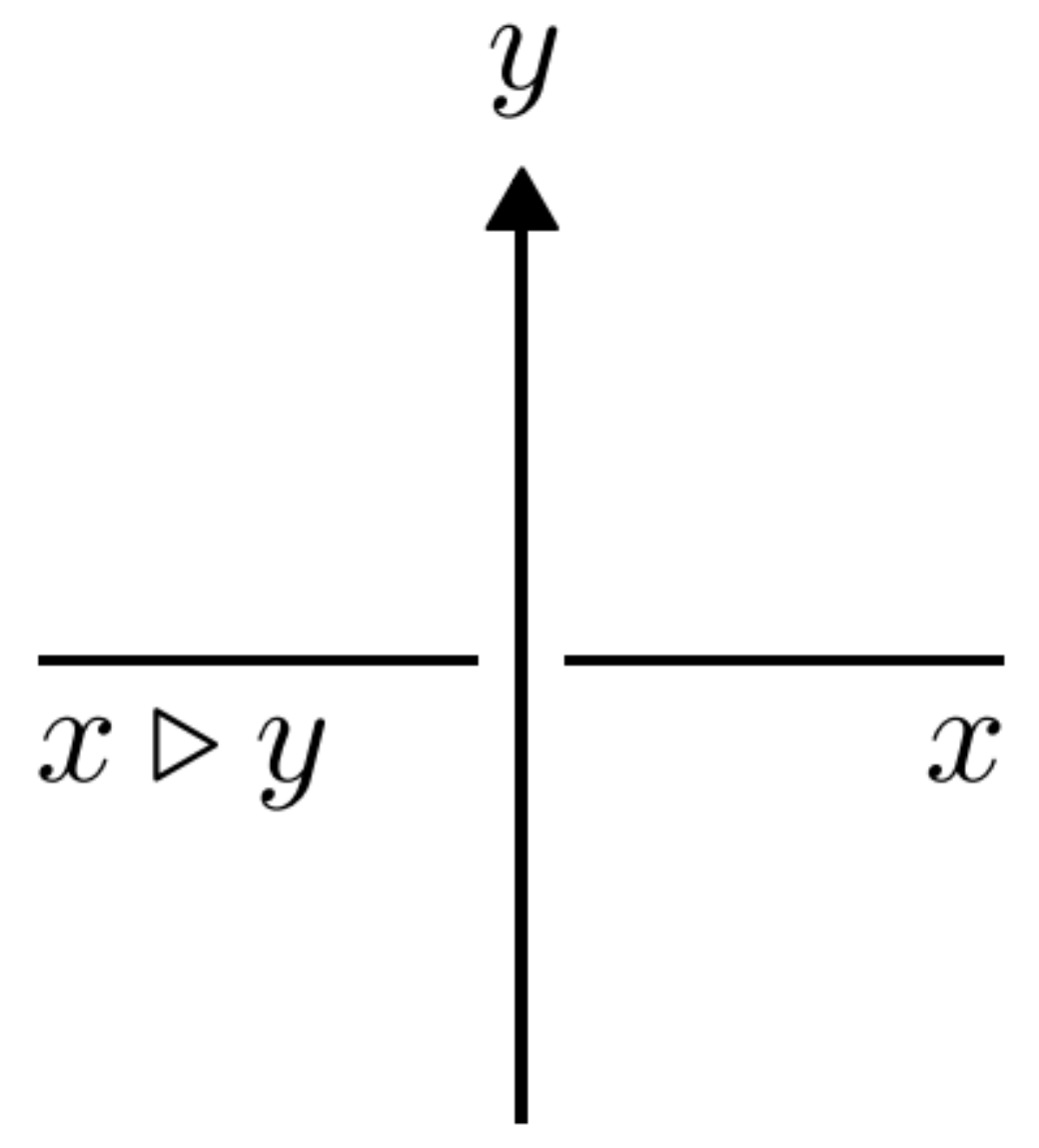}
\caption{Labelling a knot crossing in a fundamental quandle}
 \label{QuandleCrossing}
\end{figure}
\end{example}

The axioms of a quandle encode the Reidemeister moves, and thus the fundamental quandle is independent on the diagram $D$ representing the oriented link $K$. Hence, $Q(K): = Q(D)$ is a link invariant. It was shown independently by Joyce~\cite{DJoyce} and Matveev~\cite{Matveev} that the fundamental quandle is a stronger invariant than the fundamental group of a knot or link. In fact, the fundamental quandle is a complete invariant, in the sense that if $Q(K_1)$ and $Q(K_2)$ are isomorphic quandles, then the (unoriented) links $K_1$ and $K_2$ are equivalent (up to chirality). However, determining whether or not two quandles are isomorphic is not an easy task.

\begin{remark}\label{fundamental quandle of torus knot}
We consider the fundamental quandle of a $(p,2)$-torus link. By arranging the quandle relations of the fundamental quandle in a particular way, we can present $Q(\T(p,2))$ in a recursive fashion. 
We know that a $(p,2)$-torus link can be represented as the closure of the braid $\sigma_1^p$, which is also a minimal diagram for the $(p,2)$-torus link. 
There are $p$ crossings and $p$ arcs labeled $\{x_1, x_2, \dots, x_p\}$ in such a diagram of $\T(p,2)$. Considering the braid $\sigma_1^p$, we label its arcs starting from the bottom right with $x_1$  then moving to the bottom left arc to label it $x_2$, and continuing up along the braid using labels $x_3, x_4, \dots, x_p$, as shown in Fig.~\ref{FundQuandle torus knot}. We obtain the following presentation for the fundamental quandle of the $(p,2)$-torus link:

\begin{center}
$Q(\T(p,2)) =\langle x_1,x_2,x_3,\dots,x_p \, | \, x_2\kei x_1 = x_p, x_1\kei x_p = x_{p-1},$ and \\
$ \hspace{6cm} x_i \kei x_{i-1} = x_{i-2}, \text{ for all } 3 \leq i \leq p\rangle$.
\end{center}
\end{remark}

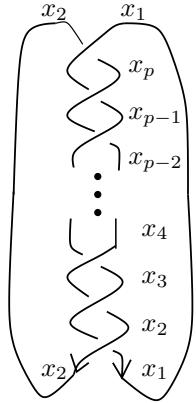
\begin{figure}[ht]
\includegraphics[height = 2in]{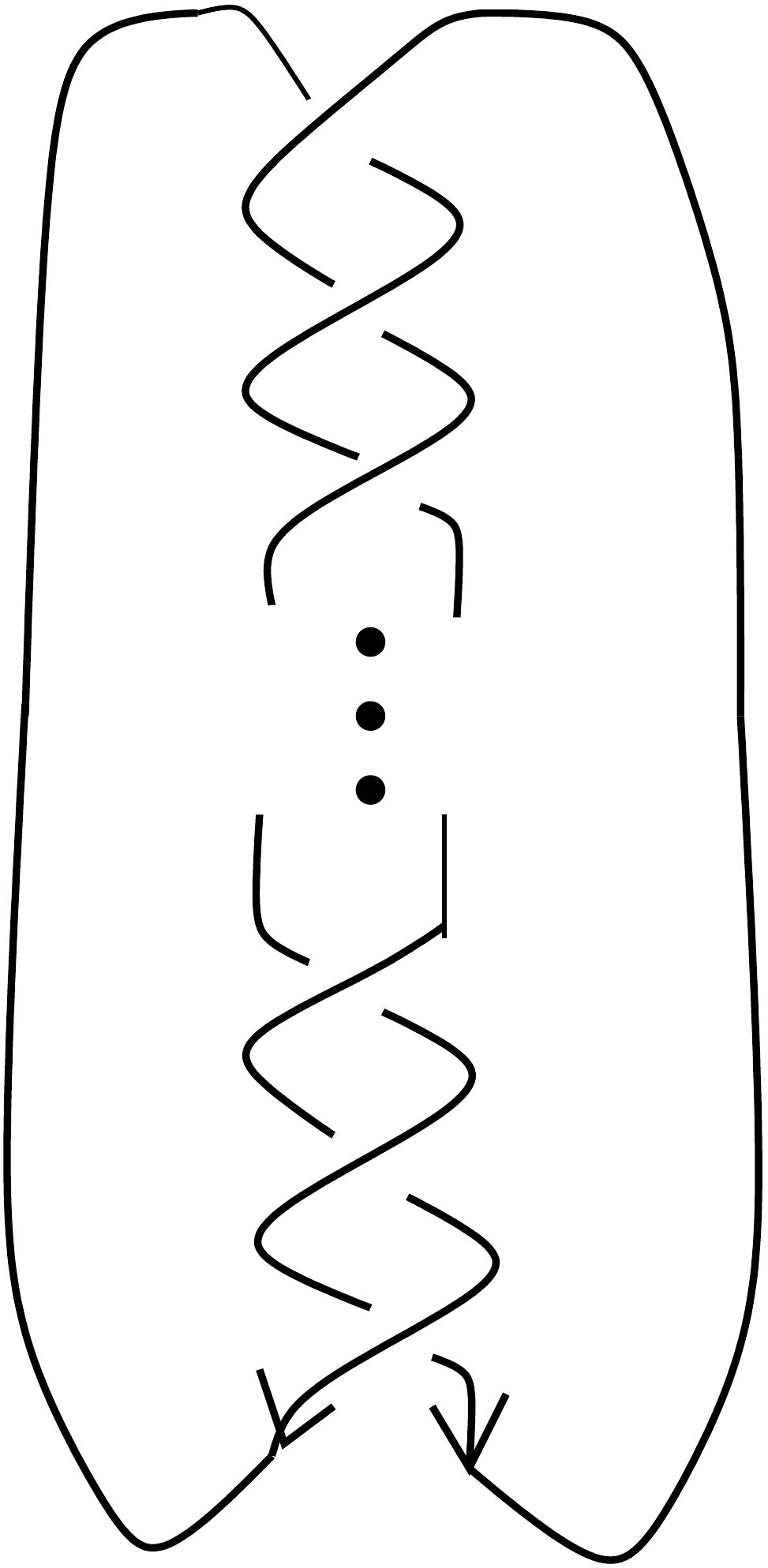}
\put(-20, 10){\fontsize{10}{10}$x_1$}
\put(-58, 10){\fontsize{10}{10}$x_2$}
\put(-20, 27){\fontsize{10}{10}$x_2$}
\put(-20, 45){\fontsize{10}{10}$x_3$}
\put(-20, 63){\fontsize{10}{10}$x_4$}
\put(-25, 90){\fontsize{10}{10}$x_{p-2}$}
\put(-25, 107){\fontsize{10}{10}$x_{p-1}$}
\put(-25, 125){\fontsize{10}{10}$x_{p}$}
\put(-28, 146){\fontsize{10}{10}$x_{1}$}
\put(-58, 146){\fontsize{10}{10}$x_{2}$}
\caption{A diagram of a $(p, 2)$-torus link}
 \label{FundQuandle torus knot}
\end{figure}

When classifying links, it is useful to count the number of quandle homomorphisms from the fundamental quandle of a link to a fixed quandle. For example, Inoue~\cite{Inoue} showed that the number of all quandle homomorphisms from the fundamental quandle of a link to an Alexander quandle~\cite{ElhamdadiNelson, Inoue} has a module structure and that it is determined by the series of the Alexander polynomials of the link.  

A \emph{quandle coloring} of an oriented link $K$ with respect to a finite quandle $\X$ (also called an \emph{$\X$-quandle coloring}) is given by a quandle homomorphism $\psi: Q(K) \to \X$. In this case, we refer to $\X$ as the coloring quandle. We let $Hom(Q(K), \X)$ be the collection of all $\X$-quandle colorings and we call it the \textit{coloring space} of $K$ with respect to $\X$. Finally, the cardinality $|Hom(Q(K), \X)|$ of the coloring space is called the \textit{quandle counting invariant with respect to $\X$} (see~\cite{ChoNelson}).

When a quandle coloring consists of a single color, we use the term \emph{trivial coloring}, and make use of the term \emph{nontrivial coloring} otherwise.

\subsection{The coloring space of $(p, 2)$-torus links with respect to dihedral quandles} \label{ssec:2.2}

In this paper we are concerned with $(p, 2)$-torus links, and we need to understand the coloring space of these links with respect to dihedral quandles. The following proposition and theorem are our first results.

\begin{proposition}\label{coloring lemma}
Let $\T(p,2)$ be a torus link and $\Z_n$ a dihedral quandle.  Consider any two fixed and consecutive arcs of $\T(p,2)$ with labels $x_1,x_2\in Q(\T(p,2))$.  If $\psi \in Hom(Q(\T(p,2)),\Z_n)$, then $p\psi(x_1)\equiv p\psi(x_2) \text{ (mod }n)$.
\end{proposition}

\begin{proof}
By Remark \ref{fundamental quandle of torus knot}, the fundamental quandle of the torus link $\T(p,2)$ has the following presentation:
\begin{center}
$Q(\T(p,2)) =\langle x_1,x_2,x_3,\dots,x_p \, | \, x_2\kei x_1 = x_p, x_1\kei x_p = x_{p-1},$ and \\
$ \hspace{6cm} x_i \kei x_{i-1} = x_{i-2}, \text{ for all } 3 \leq i \leq p\rangle$.
\end{center}
For any $\Z_n$-quandle coloring $\psi \in Hom(Q(\T(p,2)), \Z_n)$ and $x_i,x_j \in Q(\T(p,2))$, we must have that
$\psi(x_i\kei x_j) = \psi(x_i)\kei' \psi(x_j)$, where $\kei$ and $\kei'$ are the quandle operations for $Q(\T(p,2))$ and $\Z_n$, respectively.
 In particular,
\begin{align*}
\psi(x_p) &= \psi(x_2\kei x_1) = \psi(x_2)\kei'\psi(x_1) = (2\psi(x_1)-\psi(x_2))\text{ mod }n, \\
\psi(x_{p-1}) &= \psi(x_1\kei x_p) = \psi(x_1)\kei'\psi(x_p) = (2\psi(x_p)-\psi(x_1))\text{ mod }n.
\end{align*}
Moreover, for all $3 \leq i \leq p$, the following holds:
\[ \psi(x_{i-2}) = \psi(x_i\kei x_{i-1}) = \psi(x_i)\kei'\psi(x_{i-1}) = (2\psi(x_{i-1})-\psi(x_i))\text{ mod }n .\]
Using these statements inductively, we obtain:
\begin{align*}
\psi(x_2) &=(2\psi(x_3)-\psi(x_4)) \text{ mod }n \\
&=(2(2\psi(x_4)-\psi(x_5))-\psi(x_4)) \text{ mod }n \\
&= (3\psi(x_4)-2\psi(x_5))\text{ mod }n \\
&=(3(2\psi(x_5)-\psi(x_6))-2\psi(x_5)) \text{ mod }n \\
&= (4\psi(x_5)-3\psi(x_6))\text{ mod }n \\
&\hspace{1.5in}\vdots \\
&= ((p-1)\psi(x_p)-(p-2)\psi(x_1))\text{ mod } n \\
&= ((p-1)(2\psi(x_1)-\psi(x_2))-(p-2)\psi(x_1))\text{ mod } n \\
&= (p\psi(x_1)-(p-1)\psi(x_2))\text{ mod } n.
\end{align*}
It follows that $p\psi(x_2)\equiv p\psi(x_1) \text{ (mod }n).$
\end{proof}

We are now ready to prove the following theorem about the coloring space of $(p,2)$-torus links with respect to dihedral quandles, which will be used extensively in Sec.~\ref{sec:quivers}.
\begin{theorem}\label{torus coloring theorem}
For a given torus link $\T(p,2)$ and $\Z_n$ a dihedral quandle, the following holds:

(i) If $\gcd(p,n) = 1$, then $|Hom(Q(\T(p,2)), \Z_n)| = n$ and the coloring space $Hom(Q(\T(p,2)), \Z_n)$ is the collection of all trivial $\Z_n$-quandle colorings of $\T(p,2)$.

(ii)  If $\gcd(p,n) = c \neq1$, then $|Hom(Q(\T(p,2)), \Z_n)| = nc$, and the coloring space $Hom(Q(\T(p,2)), \Z_n)$ is the union of all $n$ trivial quandle colorings together with a collection of $n(c-1)$ nontrivial $\Z_n$-quandle colorings of $\T(p,2)$.
\end{theorem}

\begin{proof}
Let $\psi\in Hom(Q(\T(p,2)), \Z_n)$. By Proposition~\ref{coloring lemma}, 
\[p\psi(x_1)\equiv p\psi(x_2) \text{ (mod }n).\]

(i) If $\gcd(p,n) = 1$, then the above congruence simplifies to
\[\psi(x_1)\equiv \psi(x_2) \text{ (mod }n).\]
The latter can only occur when $\psi$ maps $x_1$ and $x_2$ to the same color. Using the relations of $Q(\T(p,2))$, as described in Remark~\ref{fundamental quandle of torus knot}, this forces the quandle homomorphism $\psi$ to color every arc in the diagram of $\T(p,2)$ the same.  Since there are $n$ quandle colorings to choose from, $|Hom(Q(\T(p,2)), \Z_n)| = n$, and $Hom(Q(\T(p,2)), \Z_n)$ is the collection of all trivial quandle colorings of $\T(p,2)$ with respect to $\Z_n$. 

(ii) If  $\gcd(p,n)=c\neq 1$, then the congruence $p\psi(x_1)\equiv p\psi(x_2) \text{ (mod }n)$ simplifies to the following:
\[ \psi(x_1)\equiv \psi(x_2) \text{ }\left(\text{mod }\frac{n}{c}\right).\]
Then $\psi(x_2) \equiv \psi(x_1)+ k\frac{n}{c}\text{ }(\text{mod }n/c)$ also holds for all integers $k$, where $0 \leq k \leq c-1$. Hence, for each fixed color $\psi(x_1) \in \Z_n$,  there are $c$ possibilities for $\psi(x_2)$ in $\Z_n$ such that $p\psi(x_1)\equiv p\psi(x_2) \text{ (mod }n)$. Since there are $n$ choices for $\psi(x_1)\in\Z_n$, we conclude that there are $nc$ $\Z_n$-quandle colorings of $\T(p,2)$.  Moreover, since there are $n$ trivial quandle colorings with respect to $\Z_n$, the remaining $n(c-1)$ $\Z_n$-quandle colorings must be nontrivial quandle colorings of the link $\T(p,2)$ with respect to $\Z_n$. 
\end{proof}

\section{Quandle coloring quivers of $(p,2)$-torus links with respect to dihedral quandles}  \label{sec:quivers}

A \textit{quandle enhancement} is an invariant for oriented links that captures the quandle counting invariant and in general is a stronger invariant. For some examples of quandle enhacements, we refer the reader to~\cite{CEGS, CJKLS, ElhamdadiNelson, GN}. More recently, by representing quandle colorings as vertices in a graph, Cho and Nelson~\cite{ChoNelson} imposed a combinatorial structure on the coloring space of a link that led to new quandle enhancements.  For example, one of these enhancements, the quandle coloring quiver, can distinguish links that have the same quandle counting invariant (see~\cite[Example 6]{ChoNelson}).

In what follows, we work with graphs that are allowed to have multiple edges and loops. To capture multiple edges in a graph $G$, we use a weight function $c \co \{(x, y) \, | \, x, y \in V(G) \} \to \N \cup \{0\}$, where $V(G)$ is the vertex set of $G$ and $(x, y)$ is any edge connecting vertices $x, y \in V(G)$. We denote the corresponding graph with a weight function $c$ by $(G, c)$. When $c$ is a constant function $c(x, y) = k$, for all $x, y \in V(G)$ and some $k \in \N$, then we use the notation $c = \hat{k}$ for the weight function and $(G, \hat{k})$ for the corresponding graph.

We say that a graph is \textit{complete}, provided that there is an edge between every pair of distinct vertices and a loop at each vertex; we denote a complete graph with $n$ vertices by $K_n$. The graphs we work with in this paper are all directed, and we use the notation $\overleftrightarrow{G}$ to denote a graph $G$ with all edges directed both ways. 

Given two graphs $G_1$ and $G_2$ with disjoint vertex sets and edge sets, the \textit{join of $G_1$ and $G_2$}, denoted $G_1 \nabla G_2$, is the graph with vertex set $V(G_1) \cup V(G_2)$ and edge set containing all edges in $G_1$ and $G_2$, as well as all edges that connect vertices of $G_1$ with vertices of $G_2$. Finally, if there is the same number of directed edges from each vertex of the graph $G_2$ to the graph $G_1$ with constant weight function $\hat{k}$ and no directed edges from vertices of $G_1$ to vertices of $G_2$, we use the notation $G_1 \overleftarrow{\nabla}_{\hat{k}} G_2$ to denote such a join of graphs.

Let $\X$ be a finite quandle and $K$ an oriented link.  For any set of quandle homomorphims $S \subset Hom(\X, \X)$, the associated \textit{quandle coloring quiver}~\cite{ChoNelson}, denoted $\Q^S_{X}(K)$, is the directed graph with a vertex $v_f$ for every element $f \in Hom(Q(K), \X)$ and an edge directed from the vertex $v_f$ to the vertex $v_g$, whenever $g = \phi  \circ f$, for some $\phi \in S$. 
When $S = Hom(\X, \X)$, we denote the associated quiver by $\Q_{\X}(K)$, and call it the \emph{full $\X$-quandle coloring quiver} of $K$.

It was explained in~\cite{ChoNelson} that $\Q^S_{X}(K)$ is a link invariant, in the sense that the quandle coloring quivers associated to diagrams that are related by Reidemeister moves are isomorphic quivers.

 \textbf{Notation.} We establish some notation for mappings.  Consider a function $g \co \X \to  \mathcal{Y}$, where $\X=\{x_1,x_2,\dots,x_n\}$ and $\mathcal{Y}=\{y_1,y_2,\dots,y_m\}$.  If $g(x_i) = y_{j_i}$, for each $x_i \in \X$, we will write $g$ as $g_{y_{j_1}, y_{j_2}, y_{j_3},\dots,y_{j_n}}$ (omitting the commas when possible). This notation will simplify the formulas for  various quandle homomorphisms considered in this paper.

Our goal is to study quandle coloring quivers of $(p,2)$-torus links with respect to dihedral quandles $\Z_n$. Before we proceed, we need to prove the following result about the space $Hom(\Z_n, \Z_n)$ of quandle automorphisms of the dihedral quandle $\Z_n$, where $n \geq 3$.

\begin{proposition}\label{Hom(X,X)}
Consider the dihedral quandle $\Z_n$, when $n \geq 3$.  Then $Hom(\Z_n, \Z_n)$ is a set of size $n^2$ of quandle homomorphisms $\phi_{\gamma} \co \Z_n \to \Z_n$ that are defined recursively, where $\gamma$ is any $n$-tuple $(\gamma_1,\gamma_2,\gamma_3,\dots, \gamma_n) \in (\Z_n)^n$, such that
\[
\phi_{\gamma}:= \left\{\begin{array}{lr}
\phi_{\gamma}(0)=\gamma_1, &\\
\phi_{\gamma}(1)=\gamma_2, &\\
\phi_{\gamma}(k) = \phi_{\gamma}(k-2)\kei \phi_{\gamma}(k-1), &\text{for }2\leq k\leq n-1.\end{array}\right.
\]
\end{proposition}

\begin{proof}
The operation of the dihedral quandle $(\Z_n, \kei)$ is given by $x \kei y = (2y-x)\text{ mod }n$, for any $x,y\in \Z_n$.  Then, $k\kei (k+1) = k+2$, for all $0 \leq k \leq n-3$, where it follows that $(n-2) \kei (n-1) = 0$ and $(n-1)\kei 0 = 1$.
Since any quandle homomorphism $\phi \in Hom(\Z_n, \Z_n)$ must satisfy $\phi(x\kei y) = \phi(x)\kei\phi(y)$, for any $x,y \in \Z_n$ (using the same $\kei$), then $\phi$ must now satisfy $\phi(k)\kei\phi(k+1) = \phi(k\kei (k+1)) = \phi(k+2)$, for all  $0 \leq k \leq n-3$, further yielding $\phi(n-2) \kei \phi(n-1) = \phi(0)$  and $\phi(n-1)\kei\phi(0) = \phi(1)$.  

Hence, $\phi(0)$ and $\phi(1)$ completely determine $\phi$, and we have the following recursive relation for the image of $(0,1,2,\dots,n-1)$ under $\phi$:
\[
(0,1,2,\dots,n-1) \stackrel{\phi}{\mapsto} (\phi(0),\phi(1),\phi(0)\kei\phi(1),\phi(1)\kei\phi(2),\dots,\phi(n-3)\kei\phi(n-2)).
\]
Since there are $n$ ways to assign each of $\phi(0)$ and $\phi(1)$ an element in $\Z_n$, we see that $|Hom(\Z_n, \Z_n)| = n^2$.
\end{proof}

\begin{remark} \label{Hom(X,X)bis}
From the cyclic behavior of the dihedral quandle $\Z_n$ and the proof of Proposition~\ref{Hom(X,X)}, we see that any homomorphism $\phi_{\gamma} \in Hom(\Z_n, \Z_n)$, where $\gamma$ is an $n$-tuple $(\gamma_1,\gamma_2,\dots, \gamma_n) \in (\Z_n)^n$, is completely determined by the images $\phi_{\gamma}(k) = \gamma_{k+1}$ and $\phi_{\gamma}(k+1) = \gamma_{k+2}$, where if $k = n-1$ then $k+2 = 0$ and if $k = n$ then $k+1 = 0$ and $k+2=1$.

We also remark that using our notation, we can write $\phi_\gamma = \phi_{\gamma_1 \gamma_2 \dots \gamma_n}$.
\end{remark}

Using Theorem~\ref{torus coloring theorem} together with Proposition~\ref{Hom(X,X)} and Corollary~\ref{Hom(X,X)bis}, we are ready to state and prove our main results regarding quivers.

\begin{theorem}\label{torus quiver theorem1}
Given a torus link $\T(p,2)$ and $\Z_n$ a dihedral quandle, where $\gcd(p,n) = 1$, the full $\Z_n$-quandle coloring quiver of $\T(p,2)$  is the complete directed graph:
$$\Q_{\Z_n}(\T(p,2)) = \left(\overleftrightarrow{K_{n}},\hat{n}\right).$$
\end{theorem}

\begin{proof}
 By Theorem~\ref{torus coloring theorem}, we know that $Hom(Q(\T(p,2)), \Z_n)$ is the collection of all $n$ trivial $\Z_n$-quandle colorings of $\T(p,2)$. Hence, the quandle coloring quiver $\Q_{\Z_n}(\T(p,2))$ has $n$ vertices, one for each of the elements in the coloring space $Hom(Q(\T(p,2)), \Z_n)$.

 Let $\phi =\phi_{\gamma_1\gamma_2\dots\gamma_n} \in Hom(\Z_n , \Z_n)$ be a quandle homomorphism given by: $0 \stackrel{\phi}{\mapsto} \gamma_1, 1 \stackrel{\phi}{\mapsto} \gamma_2, \dots,(n-1)  \stackrel{\phi}{\mapsto} \gamma_n$.
In Proposition~\ref{Hom(X,X)} we proved that $|Hom(\Z_n, \Z_n)| = n^2$, which was due to the fact that a quandle homomorphism $\phi_{\gamma_1\gamma_2\dots\gamma_n}$ is completely determined by the choices for $\gamma_1$ and $\gamma_2$. In fact, by Remark~\ref{Hom(X,X)bis}, we know that $\phi_{\gamma_1\gamma_2\dots\gamma_n}$ is determined by the choices for any two consecutive values $\gamma_k$ and $\gamma_{k+1}$, where if $k = n$ then $k+1 = 1$. 

Let $\psi_{i} \in Hom(Q(\T(p,2)), \Z_n)$ denote the trivial $\Z_n$-quandle coloring of $\T(p,2)$ associated to some fixed $i \in \Z_n$ and let $v_{\psi_i}$ be the corresponding vertex in the quandle coloring quiver $\Q_{\Z_n}(\T(p,2))$. Without loss of generality, by the above discussion, we can use that any map $\phi_{\gamma_1\gamma_2\dots\gamma_n}$ is completely determined by the choices for $\gamma_i$ and $\gamma_{i+1}$.

Consider $\phi_{\gamma_1 \dots \gamma_{i} i \gamma_{i+2 }\dots\gamma_n} \in Hom(\Z_n , \Z_n)$. That is, $\gamma_{i+1} = i$, which means that $i \stackrel{\phi}{\mapsto} i$ under this map. There are $n$ maps of the form $\phi_{\gamma_1 \dots \gamma_{i} i \gamma_{i+2 }\dots\gamma_n}$ (determined by the value of $\gamma_{i} \in \Z_n)$ and they all satisfy that 
\[ \phi_{\gamma_1 \dots \gamma_{i} i \gamma_{i+2 }\dots\gamma_n} \circ \psi_i = \psi_i.\]
 In addition, there are no other maps $\phi_{\gamma_1\gamma_2\dots\gamma_n}$ satisfying $\phi_{\gamma_1\gamma_2\dots\gamma_n} \circ \psi_i = \psi_i$. It follows that the quiver vertex $v_{\psi_i}$ has $n$ loops that represent $\psi_{i}$ being fixed by those $n$ maps.  

Let $\psi_{j} \in Hom(Q(\T(p,2)), \Z_n)$, where $j\in \Z_n$ is distinct from the fixed value of $i$ above. For similar reasons as above, there are $n$ quandle homomorphisms in $Hom(\Z_n , \Z_n)$ of the form $\phi_{\gamma_1 \dots \gamma_{i} j \gamma_{i+2 }\dots\gamma_n}$ that send $i \stackrel{\phi}{\mapsto} j$. Moreover, 
\[ \phi_{\gamma_1 \dots \gamma_{i} j \gamma_{i+2 }\dots\gamma_n} \circ \psi_i = \psi_j,\]
and there are no other maps $\phi_{\gamma_1\gamma_2\dots\gamma_n}$ satisfying $\phi_{\gamma_1\gamma_2\dots\gamma_n} \circ \psi_i = \psi_j$.
Thus $\psi_{i}$ is also mapped $n$ times to each of the colorings $\psi_{j}$. This implies that there are $n$ directed edges from the vertex $v_{\psi_i}$ to each of the vertices $v_{\psi_j}$, for all $j \not = i$.

Hence, the full quandle coloring quiver of $\T(p,2)$ with respect to $\Z_n$ is a graph with $n$ vertices where every vertex has $n$ directed edges from itself to all other vertices, including itself.  Thus, $\Q_{\Z_n}(\T(p,2))= \left(\overleftrightarrow{K_n},\hat{n}\right)$, whenever $\gcd(p,n) = 1$. 
\end{proof}

\begin{theorem}\label{torus quiver theorem2}
Let $\T(p,2)$ be a torus link and $\Z_n$ a dihedral quandle.  If $\gcd(p,n) =c$ where $c$ is prime, then the full quandle coloring quiver of $\T(p,2)$ with respect to $\Z_n$ is the join of two complete directed graphs:
\[ \Q_{\Z_n}(\T(p,2)) = \left(\overleftrightarrow{K_{n}},\hat{n}\right)  \overleftarrow{\nabla}_{\hat{d}}\left(\overleftrightarrow{K_{n(c-1)}},\hat{d}\right),\]
where $d=n/c$, and the two complete subgraphs correspond to the trivial and nontrivial, respectively, $\Z_n$-quandle colorings of $\T(p,2)$.
\end{theorem}

\begin{proof} 
Using Theorem \ref{torus coloring theorem}, the quandle coloring space $Hom(Q(\T(p,2)), \Z_n)$ is the union of all $n$ trivial $\Z_n$-quandle colorings, together with $n(c-1)$ nontrivial $\Z_n$-quandle colorings of $\T(p,2)$.  Therefore, the quandle coloring quiver $\Q_{\Z_n}(\T(p,2))$ contains $n$ vertices corresponding to the trivial $\Z_n$-quandle colorings and $n(c-1)$ vertices corresponding to the nontrivial $\Z_n$-quandle colorings of $\T(p,2)$.

Let $\psi_{\sigma}=\psi_{\sigma_1\sigma_2\sigma_3\dots\sigma_p}$ and $\psi_{\omega}=\psi_{\omega_1\omega_2\omega_3\dots\omega_p}$ be any two quandle colorings in $Hom(Q(\T(p,2)), \Z_n)$. Our notation means that $\psi_{\sigma}(x_i)=\sigma_i$  and $\psi_{\omega}(x_i)=\omega_i$, where $\sigma_i, \omega_i \in\Z_n$ and $x_i$ is a generator of the fundamental quandle $Q(\T(p,2))$.

If both $\psi_{\sigma}$ and $\psi_{\omega}$ are trivial $\Z_n$-quandle colorings (possibly the same coloring), we know from the proof of Theorem~\ref{torus quiver theorem2} that there are $n$ directed edges from the quiver vertex $v_{\psi_{\sigma}}$ to the vertex $v_{\psi_{\omega}}$ (including $n$ directed loops from the vertex $v_{\psi_{\sigma}}$ to itself, if $\psi_{\sigma}=\psi_{\omega}$). It follows that there is a complete directed graph $\left(\overleftrightarrow{K_n},\hat{n}\right)$ associated with the $n$ trivial $\Z_n$-quandle colorings of $\T(p,2)$, as a subgraph of the full quandle coloring quiver $\Q_{\Z_n}(\T(p,2))$. 

If $\psi_{\sigma}$ is a nontrivial quandle coloring, then $\psi_{\sigma}(x_1) = \sigma_1\neq \sigma_2 = \psi_{\sigma}(x_2)$, because equivalent consecutive colors in the subscript of $\psi_{\sigma}$ would impose a trivial $\Z_n$-quandle coloring. By Proposition~\ref{coloring lemma}, for any $x_1,x_2\in Q(\T(p,2))$, a quandle coloring $\psi\in Hom(Q(\T(p,2)), \Z_n)$ must satisfy the congruence $p\psi(x_1)\equiv p\psi(x_2) \text{ (mod }n)$, which reduces to $\psi(x_1)\equiv \psi(x_2) \text{ (mod }d)$, where $d= \frac{n}{c}$ and $c = \gcd(p, n), c \not = 1$. Since $\psi(x_1), \psi(x_2) \in \Z_n$, we have that for each fixed $\psi(x_1)$, there are $c$ possibilities for $\psi(x_2)$ in $\Z_n$ satisfying the above two congruences; specifically, $\psi(x_2) = \psi(x_1) + kd$, where $k$ is an integer such that $0 \leq k \leq c-1$. Moreover, since $\psi_{\sigma}(x_1) \neq \psi_{\sigma}(x_2)$ for a nontrivial quandle coloring $\psi_{\sigma}$, then for each fixed $\psi_{\sigma}(x_1)$ there are $c-1$ possibilities for $\psi_{\sigma}(x_2)$ in $\Z_n$ satisfying $p\psi_{\sigma}(x_1)\equiv p\psi_{\sigma}(x_2) \text{ (mod }n)$.

Let $\psi_{\sigma}$ be a fixed nontrivial quandle coloring in $Hom(Q(\T(p,2)), \Z_n)$. Then $\psi_{\sigma}=\psi_{\sigma_1(\sigma_1+kd)\sigma_3\dots\sigma_p}$, for some $1\leq k \leq c-1$ and $\sigma_1, \sigma_3, \dots, \sigma_p \in \Z_n$. Let $\psi_{\omega}=\psi_{\omega_1(\omega_1+hd)\omega_3\dots\omega_p}$ be any fixed coloring in $Hom(Q(\T(p,2)), \Z_n)$; that is, $h$ is an integer such that $0\leq h \leq c-1$, where $\omega_1, \omega_3, \dots, \omega_p \in \mathbb{Z}_n$ and $h \not = 0$ if $\psi_{\omega}$ is a nontrivial quandle coloring.

Now suppose $\phi_{\gamma_1\gamma_2\dots\gamma_n}\in Hom(\Z_n, \Z_n)$ satisfies $\phi_{\gamma_1\gamma_2\dots\gamma_n} \circ \psi_{\sigma} = \psi_{\omega}$. 
   In particular, this means that 
   \begin{eqnarray}
    (\phi_{\gamma_1\gamma_2\dots\gamma_n} \circ \psi_{\sigma}) (x_1) &=& \psi_{\omega}(x_1) = \omega_1 \label{eq1}\\
   (\phi_{\gamma_1\gamma_2\dots\gamma_n} \circ \psi_{\sigma}) (x_2) &=&\psi_{\omega}(x_2) = \omega_1 + hd. \label{eq2}
   \end{eqnarray}
   Since $\psi_{\sigma}(x_1) = \sigma_1$ and $\psi_{\sigma}(x_2) = \sigma_1 + kd$, equations~\eqref{eq1} and~\eqref{eq2} imply that the following must hold:
\[
    \phi_{\gamma_1\gamma_2\dots\gamma_n}(\sigma_1) = \omega_1 \,\, \text{and} \,\, \phi_{\gamma_1\gamma_2\dots\gamma_n}(\sigma_1 + kd) = \omega_1 +hd. 
\]
   
   In Remark~\ref{Hom(X,X)bis} we showed that a homomorphism $\phi_{\gamma_1\gamma_2\dots\gamma_n}\in Hom(\Z_n, \Z_n)$ is determined by two consecutive values in its subscript. Moreover, by fixing the value of $\omega_1$, a homomorphism $\phi_{\gamma_1\gamma_2\dots \omega_1 \tau \dots \gamma_n}$, with some $\tau\in\Z_n$, satisfying equation~\eqref{eq1} yields:
   \begin{eqnarray*}
   \phi_{\gamma_1\gamma_2\dots \omega_1 \tau \dots \gamma_n}(\sigma_1) &=& \omega_1\\
   \phi_{\gamma_1\gamma_2\dots \omega_1 \tau \dots \gamma_n}(\sigma_1 +1) &=& \tau \\
    \phi_{\gamma_1\gamma_2\dots \omega_1 \tau \dots \gamma_n}(\sigma_1 +2)  &=& (2 \tau - \omega_1) \text{ mod} \ n\\
     \phi_{\gamma_1\gamma_2\dots \omega_1 \tau \dots \gamma_n}(\sigma_1 +3)  &=& (3 \tau - 2\omega_1) \text{ mod} \ n \\
     \vdots \\    
     \phi_{\gamma_1\gamma_2\dots \omega_1 \tau \dots \gamma_n}(\sigma_1 + kd)  &=& (kd \tau - (kd-1)\omega_1) \text{ mod} \ n 
   \end{eqnarray*}
 Hence, a homomorphism  $\phi_{\gamma_1\gamma_2\dots \omega_1 \tau \dots \gamma_n}$ satisfies both of the equations~\eqref{eq1} and~\eqref{eq2} if and only if
 \begin{eqnarray} \label{eq3}
 kd \tau - (kd-1)\omega_1 \equiv \omega_1 + hd  \text{ (mod }n). 
 \end{eqnarray}
 
 The latter congruence is equivalent to $kd\tau \equiv kd\omega_1+hd \text{ (mod }n)$, and since $\gcd(d, n) = d$, this congruence reduces further to:
\begin{eqnarray} \label{eq4}
k\tau \equiv k\omega_1+h\text{ (mod }\frac{n}{d}).
\end{eqnarray}

Note that $k$, $\omega_1$, and $h$ are considered fixed in this congruence. Since $\frac{n}{d}$ is equal to the prime $c$ and $1 \leq k \leq c-1$, we have that $\gcd(k, \frac{n}{d}) = 1$ and thus the congruence in~\eqref{eq4} has solutions $\tau$. In particular, $\tau$ has a unique solution modulo $\frac{n}{d}$. It follows that for the congruence~\eqref{eq3}, there are $d$ incongruent solutions modulo $n$ for $\tau$.
Therefore, for a given nontrivial quandle coloring $\psi_{\sigma}$, there are $d$ homomorphisms $\phi_{\gamma_1\gamma_2 \dots \gamma_n}$ such that $\phi_{\gamma_1\gamma_2 \dots \gamma_n} \circ \psi_{\sigma} = \psi_{\omega}$, for all (trivial and nontrivial) $\Z_n$-quandle colorings $\psi_{\omega}$ of $\T(p, 2)$. 

The above reasoning shows that for each vertex $v_{\psi_{\sigma}}$ in the quiver corresponding to a nontrivial coloring, there are $d$ directed edges to all of the vertices in the quiver, including to itself. It follows that the quiver $\Q_{\Z_n}(\T(p,2))$ contains as a subgraph, a complete directed graph $\left(\overleftrightarrow{K_{n(c-1)}},\hat{d}\right)$ associated with the nontrivial quandle colorings that is joined, using the weight function $\hat{d}$, to the subgraph associated with trivial colorings:
\[ \left(\overleftrightarrow{K_n},\hat{n}\right)  \overleftarrow{\nabla}_{\hat{d}} \left(\overleftrightarrow{K_{n(c-1)}},\hat{d}\right). \]

Moreover, there is no map in $Hom(\Z_n, \Z_n)$ that sends a fixed value in $\Z_n$ to two distinct values in $\Z_n$. That is, there is no homomorphism $\phi_{\gamma_1\gamma_2 \dots \gamma_n}$ in $Hom(\Z_n, \Z_n)$ that sends a trivial $\Z_n$-quandle coloring of $\T(p,2)$ to a nontrivial $\Z_n$-quandle coloring of $\T(p,2)$. Hence, the joining behavior cannot be reciprocated from $\left(\overleftrightarrow{K_n},\hat{n}\right)$ to $\left(\overleftrightarrow{K_{n(c-1)}},\hat{d}\right)$.  

This completes the proof of $\Q_{\Z_n}(\T(p,2)) = \left(\overleftrightarrow{K_n},\hat{n}\right) \overleftarrow{\nabla}_{\hat{d}} \left(\overleftrightarrow{K_{n(c-1)}},\hat{d}\right)$, when $\gcd(p, n) = c $ and $c$ is prime. 
\end{proof}

\begin{corollary} \label{quiver-cor}
For a given torus link $\T(p,2)$ and a prime $n$, we have:
\begin{itemize}
\item If $n \nmid p$, then $\Q_{\Z_n}(\T(p,2)) = \left(\overleftrightarrow{K_{n}},\hat{n}\right)$.
\item If $n \mid p$, then $\Q_{\Z_n}(\T(p,2)) = \left(\overleftrightarrow{K_{n}},\hat{n}\right)  \overleftarrow{\nabla}_{\hat{1}}\left(\overleftrightarrow{K_{n(n-1)}},\hat{1}\right)$.
\end{itemize}
\end{corollary}

\begin{proof}
The statements follow as particular cases from Theorems~\ref{torus quiver theorem1} and~\ref{torus quiver theorem2}. Since $n$ is prime, if $n \nmid p$, then $\text{gcd}(n, p) = 1$, and Theorem~\ref{torus quiver theorem1} implies the first statement of this corollary. If $n \mid p$, then $\text{gcd}(n, p) = n$ where $n$ is prime, and the application of Theorem~\ref{torus quiver theorem2} yields the second statement above.
\end{proof}

The following statement is a direct consequence of Corollary~\ref{quiver-cor}.

\begin{corollary} 
Let $n$ be a prime.
\begin{itemize}
\item If $n \nmid p_1$ and $n \nmid p_2$, then the quandle coloring quivers $\Q_{\Z_n}(\T(p_1,2))$ and $\Q_{\Z_n}(\T(p_2,2))$ are isomorphic.
\item  If $n \mid p_1$ and $n \mid p_2$, then the quandle coloring quivers $\Q_{\Z_n}(\T(p_1,2))$ and $\Q_{\Z_n}(\T(p_2,2))$ are isomorphic.
\end{itemize}
\end{corollary}

\begin{example}
As an example, we consider the full $\Z_3$-quandle coloring quiver of the trefoil knot $\T(3, 2)$. In this case, $n = p = 3$, and according to Theorem~\ref{torus quiver theorem2}, the quiver  $\Q_{\Z_3}(\T(3,2))$ has the following form:
\[ \Q_{\Z_3}(\T(3,2)) = \left(\overleftrightarrow{K_{3}},\hat{3}\right)  \overleftarrow{\nabla}_{\hat{1}}\left(\overleftrightarrow{K_{6}},\hat{1}\right). \]

In Fig.~\ref{quiver example-part 1} on the left, we show the graph $\left(\overleftrightarrow{K_{3}},\hat{3}\right)$ associated with the trivial $\Z_3$-quandle colorings of $\T(3,2)$ and  on the right of the figure, there is the graph $\left(\overleftrightarrow{K_{6}},\hat{1}\right)$ corresponding to the $n(c-1) = 3\cdot 2$ nontrivial $\Z_3$-quandle colorings of the trefoil knot. Both of these graphs are subgraphs of the full $\Z_3$-quandle coloring quiver of the trefoil knot. In Fig.~\ref{quiver example-part 2} we present the quandle coloring quiver $\Q_{\Z_3}(\T(3,2))$ as the join of these subgraphs by the weight function $\hat{d} = \hat{1}$: $\left(\overleftrightarrow{K_{3}},\hat{3}\right)  \overleftarrow{\nabla}_{\hat{1}}\left(\overleftrightarrow{K_{6}},\hat{1}\right)$.

\begin{figure}[hb] 
 \includegraphics[height = 2in]{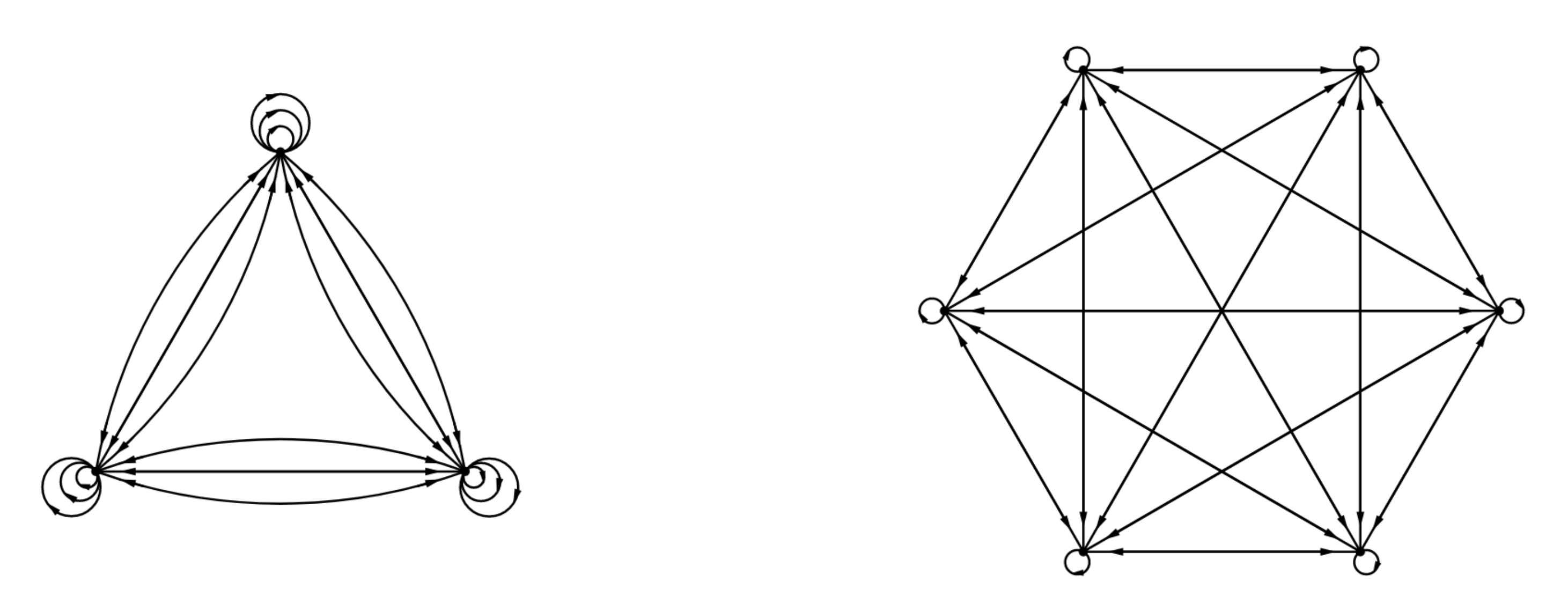}
 \put(-320, -10){\fontsize{10}{10}$\left(\overleftrightarrow{K_{3}},\hat{3}\right)$}
 \put(-100, -10){\fontsize{10}{10}$\left(\overleftrightarrow{K_{6}},\hat{1}\right)$}
 \caption{The subgraph components of the quiver $\Q_{\Z_3}(\T(3,2))$}
\label{quiver example-part 1}
\end{figure}

\begin{figure}[hb] 
\includegraphics[height = 2.6in]{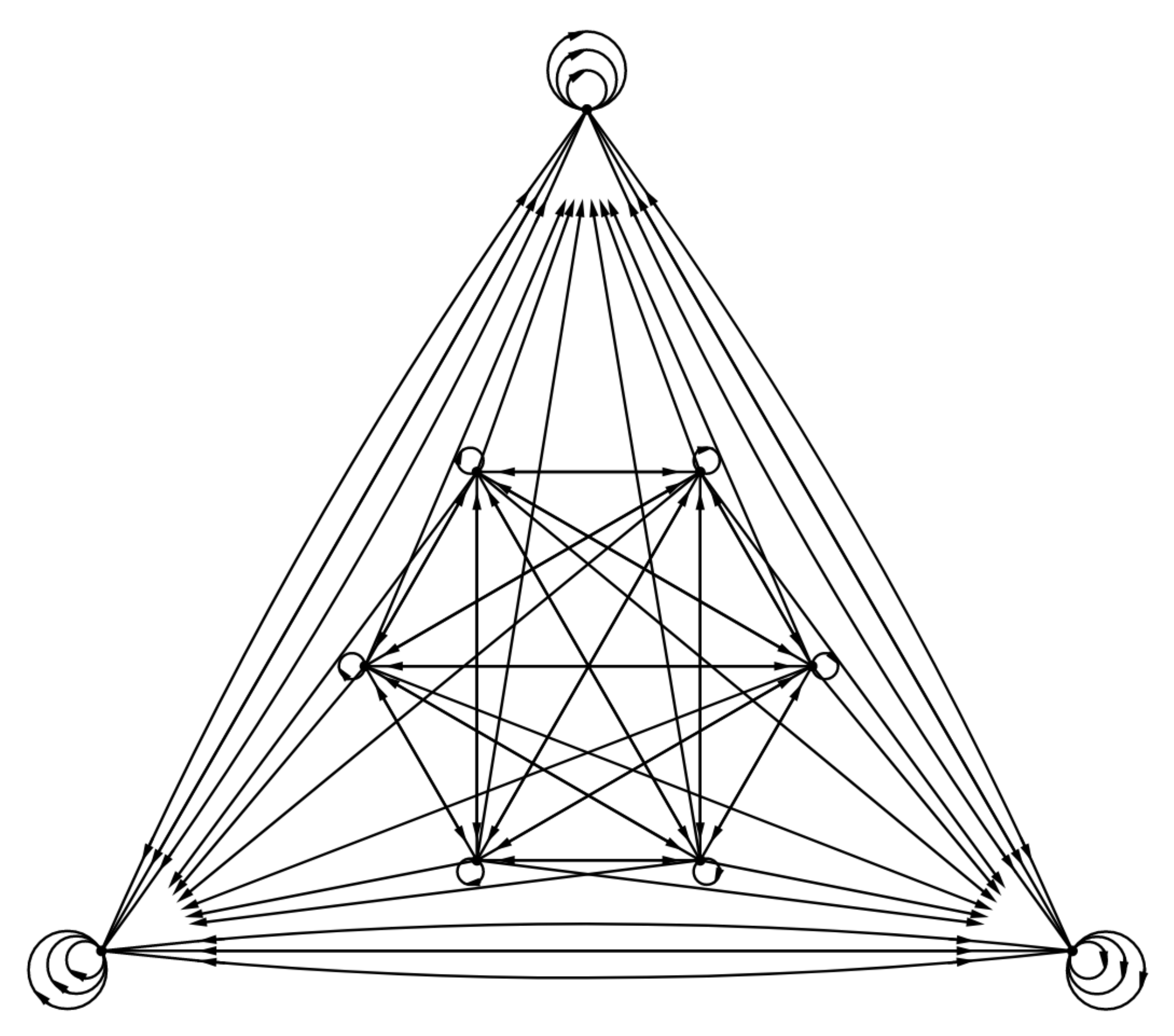}
\caption{The quandle coloring quiver $\Q_{\Z_3}(\T(3,2))$}
\label{quiver example-part 2}
\end{figure}
\end{example}

\begin{remark}
Taniguchi~\cite{Taniguchi} studied quandle coloring quivers of links with respect to dihedral quandles, where the focus was on isomorphic quandle coloring quivers. Among other things, it was proved in~\cite[Theorem 3.3]{Taniguchi} that when the dihedral quandle $\Z_n$ is of prime order, then the $\Z_n$-quandle coloring quivers of two links $L_1$ and $L_2$ are isomorphic if and only if the quandle coloring spaces of $L_1$ and $L_2$ with respect to the dihedral quandle $\Z_n$ have the same size. That is, it was proved that if $n$ is prime, then $\Q_{\Z_n}(L_1) \cong \Q_{\Z_n}(L_2) $ as quivers if and only if $|Hom(Q(L_1), \Z_n)| = |Hom(Q(L_2), \Z_n)|$. Hence, for $n$ prime, we can use Taniguchi's result in combination with our results in Theorem~\ref{torus coloring theorem} and Corollary~\ref{quiver-cor}, to say something about the full $\Z_n$-quandle coloring quiver of a certain link $L$. For this, we would need to have that $|Hom(Q(L), \Z_n)| = |Hom(Q(\T(p, 2)), \Z_n)|$, for a particular link $L$, a torus link $\T(p, 2)$, and a prime $n$. Specifically, if $n$ is prime and $|Hom(Q(L), \Z_n)|  = n$ for a certain link $L$, then $\Q_{\Z_n}(L) \cong \left(\overleftrightarrow{K_n},\hat{n}\right)$. Moreover, if $n$ is prime and $|Hom(Q(L), \Z_n)|  = n^2$, then $\Q_{\Z_n}(L) \cong  \left(\overleftrightarrow{K_n},\hat{n}\right) \overleftarrow{\nabla}_{\hat{1}} \left(\overleftrightarrow{K_{n(n-1)}},\hat{1}\right)$.
\end{remark}

\end{document}